%% file: main.tex
\documentclass[11pt]{amsart}
\usepackage[colorlinks, linkcolor=blue, citecolor=blue, pagebackref]{hyperref}
\usepackage{subcaption}
\usepackage{amssymb}
\usepackage{amsrefs}
\usepackage{mathtools}
\usepackage{ytableau}
\usepackage{extpfeil}
\usepackage[indent]{parskip}
\usepackage{stmaryrd}
\usepackage{mathrsfs}
\usepackage{fullpage}
\usepackage{nicematrix}
\usepackage{tikz}
\allowdisplaybreaks

\newcommand{\N}{\mathbb{N}}
\newcommand{\M}{\mathscr{M}}

\theoremstyle{plain}
\newtheorem{theorem}{Theorem}[section]
\newtheorem{prop}[theorem]{Proposition}
\newtheorem{lemma}[theorem]{Lemma}
\newtheorem{cor}[theorem]{Corollary}

\theoremstyle{definition}
\newtheorem{dfn}[theorem]{Definition}
\newtheorem{example}[theorem]{Example}

\theoremstyle{remark}
\newtheorem{remark}[theorem]{Remark}
\numberwithin{equation}{section}

\begin{document}

\title{The Demazure product extended to biwords}

\author{William Q.~Erickson}
\address{
William Q.~Erickson\\
Department of Mathematics\\
Baylor University \\ 
One Bear Place \#97328\\
Waco, TX 76798} 
\email{will\_erickson@baylor.edu}

\begin{abstract}

    The symmetric group $\mathfrak{S}_n$ (and more generally, any Coxeter group) admits an associative operation known as the Demazure product.
    In this paper, we first extend the Demazure product to the (infinite) set of all biwords on $\{1, \ldots, n\}$, or equivalently, the set of all $n \times n$ nonnegative integer matrices.
    We define this product diagrammatically, via braid-like graphs we call \emph{kelp beds}, since they significantly generalize the seaweeds introduced by Tiskin~(2015).
    Our motivation for this extended Demazure product arises from optimization theory, in particular the semigroup of all $(n+1) \times (n+1)$ simple nonnegative integer Monge matrices equipped with the distance (i.e., min-plus) product.
    As our main result, we show that this semigroup of Monge matrices is isomorphic to the semigroup of biwords equipped with the extended Demazure product.
    We exploit this isomorphism to write down generating functions for the growth series of the Monge matrices with respect to certain natural matrix norms.
    
\end{abstract}



\maketitle

\section{Introduction}

\subsection{The Demazure product on the symmetric group}

\label{sub:demazure}

The symmetric group $\mathfrak{S}_n$ on $n$ letters is generated by the adjacent transpositions $\sigma_1, \ldots, \sigma_{n-1}$, subject to the relations
\begin{enumerate}
    \item $\sigma_i^2 = {\rm id}$,
    \item $\sigma_i \sigma_j = \sigma_j \sigma_i$ if $|i-j|>1$,
    \item $\sigma_i \sigma_{i+1} \sigma_i = \sigma_{i+1} \sigma_i \sigma_{i+1}$.
\end{enumerate}
If we replace relation (1) by the relation $\sigma_i^2 = \sigma_i$, then we obtain an associative product on $\mathfrak{S}_n$ known as the \emph{Demazure product}, often denoted by the symbol $\star$.
(Note that this Demazure product can be defined analogously for any Coxeter system $(W,S)$ with the generators in $S$ playing the role of the $\sigma_i$.)
The underlying set of $\mathfrak{S}_n$ equipped with the Demazure product is often called the \emph{0-Hecke monoid} associated to $\mathfrak{S}_n$; see~\cite{Humphreys}*{Ch.~7} and~\cite{BB}*{\S6}.
We refer the reader to Demazure's original treatment~\cite{Demazure}*{\S5.6}; see also the contemporaneous work by Bernstein--Gelfand--Gelfand~\cite{BGG} which made use of this product.
The Demazure product typically arises in algebraic geometry in connection with Schubert varieties; see, for example, its application in obtaining the main result of~\cite{BM}.

\subsection{Biwords and kelp beds}

In this paper, we extend the Demazure product $\star$ from the symmetric group $\mathfrak{S}_n$ to the (infinite) set of all biwords with entries in $\{1, \ldots, n\}$.
This present subsection describes the theoretical interest in this extension, previewing Section~\ref{sec:Demazure}.
The next subsection (Section~\ref{sub:application}) describes the motivation arising from Monge matrices, which previews Section~\ref{sec:Monge}.
Our main result in Section~\ref{sec:main} unites these two seemingly unrelated settings.

A \emph{biword} (also called a \emph{generalized permutation}, arising commonly in classical algebraic combinatorics) is a two-row array
\[
\begin{pmatrix}
    a_1 & a_2 & \ldots & a_\ell \\
    b_1 & b_2 & \ldots & b_\ell
\end{pmatrix}
\]
such that $a_1 \leq \cdots \leq a_\ell$, and if $a_i = a_{i+1}$ then $b_i \leq b_{i+1}$.
The length $\ell$ can be any nonnegative integer.
Note that $\mathfrak{S}_n$ embeds naturally by identifying each $\sigma \in \mathfrak{S}_n$ with the biword
\[
\begin{pmatrix}
    1 & 2 & \ldots & n \\
    \sigma(1) & \sigma(2) & \ldots & \sigma(n)
\end{pmatrix}.
\]
In order to define our extended Demazure product, we visualize biwords as $(n,n)$-bipartite multigraphs, with $\ell$ edges drawn between two rows of $n$ vertices.
Then the product of two biwords is the result of stacking their graphs one above the other, and fusing certain edge pairs to obtain a third graph (see Definition~\ref{def:Demazure extended} for details).

In the original case of the Demazure product on $\mathfrak{S}_n$, a graphical interpretation has been given by Tiskin~\cite{Tiskin}, in terms of \emph{seaweed braids}.
Briefly, each element of $\mathfrak{S}_n$ corresponds to a 1-regular $(n,n)$-bipartite graph in a natural way, and Tiskin calls the edges \emph{seaweeds}; to multiply two such graphs, Tiskin stacks one over the other and then ``combs'' the $n$ seaweeds so that every pair crosses at most once. 
In the present paper, we depict biwords in an analogous graphical way; the difference, of course, is the fact that the graphs of two biwords generally contain different numbers of edges, and moreover the number of possible edges is unbounded.
To reflect the scale of this generalization from permutations to biwords, we call our graphs by the name \emph{kelp beds}, since they allow an arbitrary number of seaweeds (now called \emph{kelps}) to grow at each vertex, rather than exactly one.
% In other words, our extended Demazure product must give a way to multiply, say, three kelps by three million kelps.
The following example for $n=4$ shows the product of two biwords, along with their kelp beds depicted beneath them:
\begin{align}
\label{example in intro}
\begin{split}
    \begin{pmatrix}
        2&2&3&3&3&4\\
        1&3&3&3&3&2
    \end{pmatrix}
    \star 
    \begin{pmatrix}
        1&1&2&3\\
        3&4&4&1
    \end{pmatrix}
    &=
    \begin{pmatrix}
        2&3&4\\
        3&4&1
    \end{pmatrix} \\
    \input{Intro_example}
    \end{split}
\end{align}
(We will take up this example again in Example~\ref{ex:example from intro}, where we explain the kelp bed manipulations in detail.)

Our generalization of the Demazure product, as described above, may remind the reader of Knuth's well-known generalization~\cite{Knuth} of the Robinson--Schensted correspondence~\cite{Schensted}.
Indeed, the analogy is a close one, since both generalizations extend the symmetric group $\mathfrak{S}_n$ to the set of all biwords.
Combinatorially, the analogy can be summarized as follows, where the set of objects above each underbrace is in bijective correspondence with the set directly below: 
\begin{equation}
\label{analogy}
\underbrace{\parbox{.25\linewidth}{\centering pairs of same-shaped \\ standard tableaux\\\phantom{.}}}_{\substack{\mathfrak{S}_n \\ \text{(Robinson--Schensted)}}} : \underbrace{\parbox{.25\linewidth}{\centering pairs of same-shaped \\ semistandard tableaux \\\phantom{.}}}_{\substack{\text{biwords}\\ \text{(Knuth)}}} : :
\underbrace{\parbox{.18\linewidth}{\centering \phantom{'} \\ seaweed braids \\ \phantom{.}}}_{\substack{\mathfrak{S}_n \\ \text{(Tiskin)}}} : \underbrace{\parbox{.1\linewidth}{\centering \phantom{'} \\ kelp beds \\ \phantom{.}}}_{\substack{\text{biwords}\\ \text{(this paper)}}}
\end{equation}
In the analogy displayed above, the structure which is extended on the left-hand side is the row insertion algorithm that endows the Robinson--Schensted correspondence with its many notable properties; the structure which is extended on the right-hand side is, as we will show in this paper, the Demazure product on $\mathfrak{S}_n$, and along with it a certain connection to the min-plus matrix product.

% There is really nothing canonical about our extension of the Demazure product in this paper; rather, we have defined it with a specific application in mind (see Section~\ref{sub:application} just below).
We should point out another, quite different, extension of the Demazure product which is introduced in a recent preprint by Pflueger~\cite{Pflueger}.
In that paper, the Demazure product is extended not to biwords, but to a certain set of integer permutations $\sigma : \mathbb{Z} \longrightarrow \mathbb{Z}$ called \emph{almost sign-preserving permutations}.
In all three papers, however --- namely, Tiskin~\cite{Tiskin}, Pflueger~\cite{Pflueger}, and the present paper --- a major motivation for studying the Demazure product is its connection to the distance product (i.e., the min-plus product) of matrices, as explained below.

\subsection{The distance product of Monge matrices}
\label{sub:application}

An important role in optimization theory is played by matrices $A$ with the \emph{Monge property}:
\[
    A_{ij} + A_{IJ} \leq A_{iJ} + A_{Ij}, \qquad \text{for all $i < I$ and $j < J$}.
\]
The term was coined by Hoffman in his work~\cite{Hoffman} proving that the transportation problem can be solved by the greedy algorithm known as the ``northwest corner rule'' if and only if the associated cost matrix has this Monge property; the namesake, 18th-century geometer Gaspard Monge, was the first to consider the transportation problem~\cite{Monge} and is widely regarded as the originator of optimal transport theory~\cite{Villani}.
We refer the reader to the excellent survey article~\cite{Burkard} describing many applications of Monge matrices, along with the more recent references contained in~\cite{Tiskin} relating to graph and string algorithms; see also~\cite{Russo}, for example.

In many applications of Monge matrices, it is most natural to consider the \emph{distance product}, denoted by $\odot$, rather than the ordinary matrix product.
The distance product is also known as the \emph{tropical} or the \emph{min-plus} product; this is obtained from the ordinary matrix product by replacing ``+'' with ``min,'' and replacing ``$\times$'' with ``+''.
The name ``distance product'' refers to the fact that if $A$ is the adjacency matrix of a weighted graph, then the entries of $A^{\odot k}$ give the distances between vertices using paths of length at most $k$.
Tiskin~\cite{Tiskin}, while studying faster algorithms for computing the distance product of Monge matrices, observed that $(\mathfrak{S}_n, \star)$ is isomorphic as a monoid to a set of certain special Monge matrices equipped with the distance product $\odot$.
In particular, Tiskin associates to each element $\sigma \in \mathfrak{S}_n$ a ``simple unit-Monge matrix'' which is its image under this isomorphism of monoids.

The main result of our paper (Theorem~\ref{thm:main result}) removes Tiskin's ``unit'' condition, thus extending his isomorphism to a semigroup isomorphism $\Phi$ between the set of biwords equipped with the extended Demazure product $\star$, and the set of simple Monge matrices of a given dimension equipped with the distance product $\odot$.
(We follow Tiskin in calling a Monge matrix \emph{simple} if its first column and bottom row are all 0's.)
In other words, stacking and fusing two of our kelp beds (as described above) is nothing other than computing the distance product of their associated Monge matrices (i.e., their images under $\Phi$).
We emphasize, as with our Robinson--Schensted--Knuth comparison before, that this generalization expands a finite set (the set of $n!$ many simple \emph{unit}-Monge matrices) to an infinite one (the set of all $(n+1) \times (n+1)$ simple Monge matrices with nonnegative integer entries).
In fact, one could extend the analogy~\eqref{analogy} as follows:
\[
\ldots \text{\eqref{analogy}} \ldots : \: : 
\underbrace{\parbox{.3\linewidth}{\centering \phantom{'} \\ simple unit-Monge matrices \\ \phantom{.}}}_{\substack{\mathfrak{S}_n \\ \text{(Tiskin)}}} : \underbrace{\parbox{.25\linewidth}{\centering \phantom{'} \\ simple Monge matrices \\ \phantom{.}}}_{\substack{\text{biwords}\\ \text{(this paper)}}}
\]

Tiskin's motivation in~\cite{Tiskin} was computational: viewing simple unit-Monge matrices as elements of $\mathfrak{S}_n$ allows him to give an algorithm for their distance product with runtime $O(n \log n)$, improving on the previously known algorithm with runtime $O(n^2)$.
We have not yet determined whether our extension to all simple Monge matrices will retain a similar computational advantage in general --- we are somewhat doubtful of this --- and we would be interested in the expertise of specialists in computer science and algorithms.

As secondary results (and actually as an application of our isomorphism $\Phi$), we give closed-form generating functions for the growth series of the simple Monge matrices, with respect to two typical matrix norms; see Theorem~\ref{thm:growth series}.
Computing these series leads to some interesting bijections with other objects in the combinatorial literature, which we collect in Propositions~\ref{prop:plane partitions} and~\ref{prop:interpretations}.

\begin{remark}
    \label{rem:biwords}
    Throughout the paper, to keep notation compact and concrete, we will usually view biwords as $n \times n$ matrices $X$ with nonnegative integer entries, where the entry $X_{ij}$ is the number of times the column $(i,j)^T$ appears in the biword.
    (Equivalently, the matrix $X$ is the biadjcency matrix of the kelp bed of the biword.)
    Hence, glancing through the paper, the reader will primarily see us working with this set of matrices, which we denote by $\N^{n \times n}$.
    In this incarnation, the group $\mathfrak{S}_n$ embeds as the permutation matrices.
    Philosophically, however, we somewhat prefer the language of ``biwords'' (i.e., generalized permutations), which is why we have emphasized that perspective in this introduction.
\end{remark}

\section{The Demazure product extended to biwords}
\label{sec:Demazure}

\subsection{Matrices, biwords, and kelp beds}

Fix a positive integer $n$, and let $\N^{n \times n}$ denote the set of all $n \times n$ matrices with entries in $\N \coloneqq \mathbb{Z}_{\geq 0}$.
As usual, if $X \in \N^{n \times n}$, then we write $X_{ij}$ for the $(i,j)$ entry of $X$.
We write $[n] \coloneqq \{1, \ldots, n\}$.
Of course, rather than viewing $\N^{n \times n}$ as the set of all $\N$-valued functions on $[n] \times [n]$, one can instead view elements of $\N^{n \times n}$ as (finite) multisets with elements taken from $[n] \times [n]$.
From this perspective, $X$ is the multiset
\begin{equation}
    \label{X as multiset}
    X = \Big\{ \underbrace{(i,j), \ldots, (i,j)}_{X_{ij} \text{ copies}} : 1 \leq i,j \leq n \Big\}.
\end{equation}
With the introduction (and Remark~\ref{rem:biwords}) in mind, we point out that if one lists off the elements of $X$ in lexicographical order, writing each copy of $(i,j)$ as a column vector, then one obtains a \emph{biword} in the classical sense.
Since from now on we will speak of matrices in $\N^{n \times n}$ rather than ``biwords,'' we take the opportunity now to remark that there is no essential difference between them.

An element $X \in \N^{n \times n}$ can also be viewed as the biadjacency matrix of a labeled $(n,n)$-bipartite graph, depicted as follows.
We draw two rows of vertices, with each row labeled $1, \ldots, n$ from left to right, and we draw $X_{ij}$ many edges $(i,j)$ from vertex $i$ in the top row to vertex $j$ in the bottom row.
Hence the total number of edges equals the sum of the entries in $X$.
We depict an edge as a curve whose vertical projection is always directed downward.
By analogy with the work of Tiskin~\cite{Tiskin}, explained in the introduction, we call the edges \emph{kelps}, and we call the graph itself the \emph{kelp bed} of the matrix $X$.

When viewing matrices as their kelp beds, we will often use the multiset convention~\eqref{X as multiset}, whereby $X$ is the multiset of its kelps.
For example, when we write  ``Let $x \in X$,'' we mean ``Let $x$ be some kelp $(i,j)$ appearing at least once in the kelp bed of $X$.''
In particular, there are $X_{ij}$ many distinguishable copies of each kelp $(i,j)$.
Likewise, $X \cup \{(i,j)\}$ is always understood as a sum of multisets, or equivalently, as the matrix obtained from $X$ by replacing $X_{ij}$ with $X_{ij} + 1$.
We write $\#X$ to denote the cardinality of $X$ as a multiset; equivalently, $\#X$ is the number of kelps in the kelp bed of $X$, and also the sum of the entries of the matrix $X$.

\subsection{The Demazure product extended to \texorpdfstring{$\N^{n \times n}$}{the matrices}}

Recall the Demazure product $\star$ on the symmetric group $\mathfrak{S}_n$, described in Section~\ref{sub:demazure}.
In this subsection, we will extend this Demazure product to all of $\N^{n \times n}$ by manipulating kelp beds.
Roughly speaking, we will stack two kelp beds, reduce to a certain subdiagram of that stack, and then fuse certain pairs of kelps to obtain a third kelp bed.
We begin with two important definitions (Definitions~\ref{def:up-down} and~\ref{def:tangledness}) regarding the notion of an \emph{up--down pair} in a stack of two kelp beds. 

\begin{dfn}[Up--down pairs]
    \label{def:up-down}
    Let $X,Y \in \N^{n \times n}$.
    An \emph{up--down pair} in $(X,Y)$ is a pair $(x,y) \in X \times Y$, where $x = (a,b)$ and $y = (c,d)$, such that $b \leq c$.
\end{dfn}

Given $X,Y \in \N^{n \times n}$, viewed as kelp beds, it will be convenient to visualize $(X,Y)$ as the graph obtained by identifying the bottom vertices of $X$ with the top vertices of $Y$.
The use of the term ``up--down pair'' for $(x,y)$ is due to the fact that as one scans the middle row of vertices in $(X,Y)$ from left to right, one sees the kelp $x$ (which ``grows'' upward) before the kelp $y$ (which ``grows'' downward).

\begin{dfn}[Weight]
\label{def:tangledness}
    Let $X,Y \in \N^{n \times n}$.
    A \emph{system of up--down pairs} for $(X,Y)$ is a multiset $\{(x_i, y_i)\}$ of up--down pairs, such that $\bigcup_i x_i \subseteq X$ and $\bigcup_i y_i \subseteq Y$ as multisets.
    % In other words, the kelp $(i,j)$ appears at most $X_{ij}$ (resp., $Y_{ij}$) times among the $x_i$ (resp., the $y_i$).
    The \emph{weight} of $(X,Y)$, denoted by ${\rm wt}(X,Y)$ is the cardinality of the largest system of up--down pairs for $(X,Y)$.
    \end{dfn}

We are now ready for the central definition of this paper.

\begin{dfn}[Demazure product on $\N^{n \times n}$]
\label{def:Demazure extended}
    Let $X,Y \in \N^{n \times n}$.
    The \emph{Demazure product} $X \star Y \in \N^{n \times n}$ is defined via kelp beds as follows:

    \begin{enumerate}
       
        \item Set $X_0 = Y_0 = \varnothing$. \\
        Starting at $\ell=1$, construct the kelp sub-beds $X_\ell \subseteq X$ and $Y_\ell \subseteq Y$ iteratively as follows:

        \begin{enumerate}
        
        \item Let $x_\ell \in X$ be the rightmost kelp (with respect to the top vertex, then the bottom to break ties) whose top vertex is weakly left of all the kelps in $X_{\ell-1}$.
        If no such $x_\ell$ exists, then the construction terminates at $\ell-1$.
        
        \item Let $y_\ell \in Y$ be the leftmost kelp (with respect to bottom vertex, then top to break ties) such that ${\rm wt}(X_{\ell-1} \cup \{x_\ell\}, \: Y_{\ell-1} \cup \{y_\ell\}) > {\rm wt}(X_{\ell-1}, Y_{\ell-1})$.
        If no such $y_\ell$ exists, then return to step (a) and replace $x_\ell$ by the next kelp in $X$ to its left.

        \item Set $X_\ell = X_{\ell-1} \cup \{x_\ell\}$ and $Y_\ell =  Y_{\ell-1} \cup \{y_\ell\}$.
        Then increment $\ell$ by 1 and return to step (a).

        \end{enumerate}

        \item Let $w$ be the index at which the construction above terminates.
        Then $X \star Y$ is obtained from $(X_w, Y_w)$ by deleting the middle row of vertices, and fusing each pair $(x_\ell, y_\ell)$ into a single kelp.
    \end{enumerate}
    
    \end{dfn}

\begin{example}
\label{ex:example from intro}
    We reexamine the example~\eqref{example in intro} from the introduction, where $n=4$.
    Let $X$ and $Y$ be the following matrices (shown along with their kelp beds):
    \[
    X = \begin{bmatrix}
        0 & 0 & 0 & 0 \\
        1 & 0 & 1 & 0 \\
        0 & 0 & 3 & 0 \\
        0 & 1 & 0 & 0
    \end{bmatrix}
    = \begin{tikzpicture}[scale=.5,every node/.style={scale=.75},baseline=5pt]
    \def\n{4}
    \foreach \i in {1,...,\n} {
    \node[draw, fill=black, circle, inner sep=1.5pt] (a\i) at (\i, 1) {};
    \node[above=2pt] at (a\i) {\i};
}
    \foreach \i in {1,...,\n} {
    \node[draw, fill=black, circle, inner sep=1.5pt] (b\i) at (\i, 0) {};
    \node[below=2pt] at (b\i) {\i};
}
\draw[thick] (a2) .. controls +(0,-0.5) and +(0,0.5) .. (b1);
\draw[thick] (a2) .. controls +(0,-0.5) and +(-.5,0) .. (b3);
\draw[thick] (a3) .. controls +(0,-0.5) and +(0,0.5) .. (b3);
\draw[thick] (a3) .. controls +(.25,-0.5) and +(.25,0.5) .. (b3);
\draw[thick] (a3) .. controls +(-.25,-0.5) and +(-.25,0.5) .. (b3);
\draw[thick] (a4) .. controls +(0,-0.5) and +(0,0.5) .. (b2);
\end{tikzpicture}
    \qquad\qquad
    Y = \begin{bmatrix}
        0 & 0 & 1 & 1 \\
        0 & 0 & 0 & 1 \\
        1 & 0 & 0 & 0 \\
        0 & 0 & 0 & 0
    \end{bmatrix}
    =
    \begin{tikzpicture}[scale=.5,every node/.style={scale=.75},baseline=5pt]

% Number of nodes
\def\n{4}

% Draw the first row of nodes
\foreach \i in {1,...,\n} {
    \node[draw, fill=black, circle, inner sep=1.5pt] (a\i) at (\i, 1) {};
    \node[above=2pt] at (a\i) {\i};
}

% Draw the second row of nodes
\foreach \i in {1,...,\n} {
    \node[draw, fill=black, circle, inner sep=1.5pt] (b\i) at (\i, 0) {};
    \node[below=2pt] at (b\i) {\i};
}

\draw[thick] (a1) .. controls +(0,-0.5) and +(-.5,0.5) .. (b3);
\draw[thick] (a1) .. controls +(0.5,-0.5) and +(-.5,+0.5) .. (b4);
\draw[thick] (a2) .. controls +(0.5,-0.5) and +(0,0.5) .. (b4);
\draw[thick] (a3) .. controls +(0,-0.5) and +(0,0.5) .. (b1);

\end{tikzpicture}
    \]
We visualize the ordered pair $(X,Y)$ by stacking the kelp beds as follows (where we suppress the vertex labels to reduce clutter):
\[
(X,Y) = \:
\begin{tikzpicture}[scale=.5,every node/.style={scale=.75},baseline=-3pt]
    \def\n{4}
    \foreach \i in {1,...,\n} {
    \node[draw, fill=black, circle, inner sep=1.5pt] (a\i) at (\i, 1) {};
}
    \foreach \i in {1,...,\n} {
    \node[draw, fill=black, circle, inner sep=1.5pt] (b\i) at (\i, 0) {};
}
    \foreach \i in {1,...,\n} {
    \node[draw, fill=black, circle, inner sep=1.5pt] (c\i) at (\i, -1) {};
}
\draw[thick] (a2) .. controls +(0,-0.5) and +(0,0.5) .. (b1);
\draw[thick] (a2) .. controls +(0,-0.5) and +(-.5,0) .. (b3);
\draw[thick] (a3) .. controls +(0,-0.5) and +(0,0.5) .. (b3);
\draw[thick] (a3) .. controls +(.25,-0.5) and +(.25,0.5) .. (b3);
\draw[thick] (a3) .. controls +(-.25,-0.5) and +(-.25,0.5) .. (b3);
\draw[thick] (a4) .. controls +(0,-0.5) and +(0,0.5) .. (b2);
\draw[thick] (b1) .. controls +(0,-0.5) and +(-.5,0.5) .. (c3);
\draw[thick] (b1) .. controls +(0.5,-0.5) and +(-.5,+0.5) .. (c4);
\draw[thick] (b2) .. controls +(0.5,-0.5) and +(0,0.5) .. (c4);
\draw[thick] (b3) .. controls +(0,-0.5) and +(0,0.5) .. (c1);

\end{tikzpicture}
\]
We now compute the product $X \star Y$ by using Definition~\ref{def:Demazure extended}.
Following step (1), we obtain the following sub-beds $(X_\ell, Y_\ell)$, for $1 \leq \ell \leq w = 3$.
For each $\ell$, we highlight the kelps $x_\ell$ and $y_\ell$ in a distinct color (blue for $\ell=1$, red for $\ell = 2$, and green for $\ell = 3$):

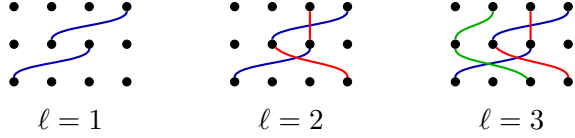
\input{Ex_after_definition}
\end{example}

\noindent Now following step (2), we take the final pair $(X_3, Y_3)$ shown above, delete its middle row of vertices, and fuse together the kelp pairs of each color to obtain the product
\[
X \star Y = \begin{tikzpicture}[scale=.5,every node/.style={scale=.75},baseline=5pt]

% Number of nodes
\def\n{4}

% Draw the first row of nodes
\foreach \i in {1,...,\n} {
    \node[draw, fill=black, circle, inner sep=1.5pt] (a\i) at (\i, 1) {};
    \node[above=2pt] at (a\i) {\i};
}

% Draw the second row of nodes
\foreach \i in {1,...,\n} {
    \node[draw, fill=black, circle, inner sep=1.5pt] (b\i) at (\i, 0) {};
    \node[below=2pt] at (b\i) {\i};
}

\draw[thick,green!70!black] (a2) .. controls +(0,-0.5) and +(0,0.5) .. (b3);
\draw[thick,red] (a3) .. controls +(0,-0.5) and +(0,+0.5) .. (b4);
\draw[thick,blue!70!black] (a4) .. controls +(0,-0.5) and +(0,0.5) .. (b1);

\end{tikzpicture}
= \begin{bmatrix}
    0 & 0 & 0 & 0 \\
    0 & 0 & 1 & 0\\
    0 & 0 & 0 & 1\\
    1 & 0 & 0 & 0
\end{bmatrix}.
\]

Of course, in order to justify the term ``extended Demazure product,'' we should verify that upon restricting $\N^{n \times n}$ to the symmetric group $\mathfrak{S}_n$ (realized here as the permutation matrices), the product $\star$ does indeed recover the Demazure product on $\mathfrak{S}_n$.
This will follow directly from our main result (see the proof of Theorem~\ref{thm:main result}), combined with either Tiskin's~\cite{Tiskin}*{Thm.~4} or Pflueger's~\cite{Pflueger}*{eqn.~(2)} observation regarding the special case of $\mathfrak{S}_n$.

\section{The distance product of simple Monge matrices}
\label{sec:Monge}

In this section, we shift perspective to optimization theory, where we consider the set $\M^0_n$ consisting of all ${(n+1) \times (n+1)}$ nonnegative integer matrices which have the \emph{(simple) Monge property}.
In applications, it is typical to study the distance product of Monge matrices, denoted by $\odot$, and also known as the \emph{tropical product} or the \emph{min-plus product}.
This product endows $(\M^0_n, \odot)$ with a semigroup structure.
Since it is natural in this setting to consider both the max norm and the $L_{1,1}$ norm (which are integer-valued on $\M^0_n$), we give generating functions (Theorem~\ref{thm:growth series}) to enumerate the matrices in $\M^0_n$ with a given norm.
Our main method in doing this is to associate each Monge matrix to a matrix in $\N^{n \times n}$, namely its preimage under a certain bijection $\Phi$ (which will be the subject of our main result in Theorem~\ref{thm:main result}).
These generating functions can be viewed as a sort of growth series for $\M^0_n$ with respect to the given matrix norms.

\subsection{Monge matrices}

An $(n + 1) \times (n+1)$ nonnegative integer-valued matrix is said to have the \emph{Monge property} if it belongs to the set
\begin{equation}
    \label{Mn}
    \M_n \coloneqq \Big\{ A \in \N^{(n+1) \times (n+1)} : A_{ij} + A_{IJ} \leq A_{iJ} + A_{Ij} \text{ for all $i<I$ and $j<J$} \Big\}.
\end{equation}
(The reason for the $n+1$ will become apparent in~\eqref{M0n} below.)
It is well known, and straightforward to see by induction~\cite{Burkard}*{eqn.~(6)}, that $A$ is Monge if and only if every $2 \times 2$ block is Monge; hence the condition in~\eqref{Mn} may be replaced by
\begin{equation}
    \label{Monge condition contiguous}
    A_{ij}+A_{i+1, j+1} \leq A_{i,j+1} + A_{i+1,j} \text{ for all } 1 \leq i,j \leq n.
\end{equation}
We define the map ${\Delta : \M_n \longrightarrow \N^{(n+1) \times (n+1)}}$ sending each Monge matrix $A$ to its \emph{density matrix} $\Delta(A)$ given by
\[
\Delta(A)_{ij} = \begin{cases}
    A_{ij} + A_{i+1, j-1} - A_{i, j-1} - A_{i+1, j}, & i \neq n+1 \text{ and } j \neq 1,\\
    0 & \text{otherwise}.
\end{cases}
\]
As a toy example where $n=2$, we have
\[
\begin{bmatrix}
    8 & 5 & 6 \\
    7 & 3 & 1 \\
    13 & 5 & 1
\end{bmatrix}
\xmapsto{\Delta}
\begin{bmatrix}
    0 & 1 & 3 \\
    0 & 4 & 2 \\
    0 & 0 & 0 
\end{bmatrix}.
\]
Note that the nonzero entries of $\Delta(A)$ lie in its $n \times n$ upper-right block.
Nearly inverse to the map $\Delta$ is the map $\Sigma : \N^{(n+1) \times (n+1)} \longrightarrow  \M_n$ which sends a matrix $B$ to its \emph{distribution matrix} $\Sigma(B)$ given by
\begin{equation}
    \label{A Sigma}
    \Sigma(B)_{ij} = \sum_{\mathclap{\substack{i' \geq i,\\ j'\leq j \phantom{,}}}} B_{i' j'}.
\end{equation}
That is, each entry of $\Sigma(B)$ is the sum of all the entries in $B$ lying weakly southwest of the corresponding matrix position.
As another toy example where $n=2$, we have
\[
\begin{bmatrix}
    3 & 0 & 2\\
    1 & 4 & 0\\
    2 & 3 & 1
\end{bmatrix}
\xmapsto{\Sigma}
\begin{bmatrix}
   6 & 13 & 16 \\
   3 & 10 & 10 \\
   2 & 5 & 6
\end{bmatrix}.
\]

\subsection{The semigroup \texorpdfstring{$(\M^0_n, \odot)$}{} of simple Monge matrices}

We observe that $\Delta$ and $\Sigma$ are, in fact, mutually inverse if we restrict our attention to \emph{simple} Monge matrices, defined as follows.
Following Tiskin~\cite{Tiskin}*{Def.~3}, we say that a Monge matrix is \emph{simple} if its leftmost column and bottommost row consist of all zeros; we denote the subset of simple $(n+1) \times (n+1)$ Monge matrices by
\begin{equation}
    \label{M0n}
    \M^0_n \coloneqq \Big\{ A \in \M_n : A_{ij} = 0 \text{ if } i = n+1 \text{ or } j = 1 \Big\}.
\end{equation}
In retrospect, this explains the subscript $n$ (rather than $n+1$) which we used to decorate the notation in both~\eqref{Mn} and~\eqref{M0n}:
a simple $(n +1) \times (n+1)$ Monge matrix may as well be viewed as an $n \times n$ matrix which is padded with zeros along its lower-left boundary.
To make this precise, define the embedding
\begin{align}
    \label{L}
    \begin{split}
        {\rm L} : \N^{n \times n} & \longrightarrow \N^{(n+1) \times (n+1)},\\
        X & \longmapsto
        \scalebox{.5}{$\begin{bNiceArray}{c|ccc}
  0 & \Block{3-3}<\Huge>{X} \\
  \raisebox{2pt}{\vdots} \\
  0 \\
  \hline
  0 & 0 & \cdots & 0
\end{bNiceArray}$}
        % \begin{bmatrix}
        %     0 & M \\
        %     0 & 0
        % \end{bmatrix}
    \end{split}
\end{align}
by padding an $n \times n$ matrix with zeros along its lower-left boundary (in the shape of an L).
We then have the four bijections in the following diagram, where (upon restriction to the (co)domains below) $\Sigma$ and $\Delta$ are mutually inverse:

\input{Big_bijection_diagram}

\noindent The composition $\Phi \coloneqq \Sigma \circ {\rm L}$ will be the subject of our main result (Theorem~\ref{thm:main result}), and so we record below an explicit formula to obtain the simple Monge matrix $\Phi(X)$ from $X$.

\begin{lemma}
\label{lemma:Phi}
Let $X \in \N^{n \times n}$, and let $\Phi : \N^{n \times n} \longrightarrow \M^0_n$ be the bijection depicted in~\eqref{dist density bijection}.
We have
\[
    \Phi(X) = \sum_{i,j=1}^n X_{i,-j} \underbrace{\left[\begin{array}{c|c}
        0 & \mathbf{1}_{i \times j} \\
        \hline
        0 & 0
    \end{array}\right]}_{(n+1) \times (n+1)},
\]
where the index $-j$ denotes the $j$th column from the right, and $\mathbf{1}_{i \times j}$ is the $i \times j$ matrix in which every entry is 1.
\end{lemma}

\begin{proof}
    Let $E_{ij}$ denote the $(n+1) \times (n+1)$ matrix with 1 in the $(i,j)$ position and zeros elsewhere.
    We have
    \begin{align*}
        \Phi(X)  &= \sum_{i,j=1}^{n+1} \Phi(X)_{ij} E_{ij}\\
        &= \sum_{i,j=1}^{n+1} 
        \Sigma \left( {\rm L}(X) \right)_{ij} E_{ij} & \text{since $\Phi \coloneqq \Sigma \circ {\rm L}$ in~\eqref{dist density bijection}} \\
        &= \sum_{i,j=1}^{n+1} \Bigg(\sum_{\substack{i' \geq i, \\ j' \leq j \phantom{,}}} {\rm L}(X)_{i'j'}\Bigg) E_{ij} & \text{by~\eqref{A Sigma}}\\
        &= \sum_{i,j=1}^{n+1} \Bigg(\sum_{\substack{i' \geq i, \\ j' \geq j \phantom{,}}} {\rm L}(X)_{i',-j'}\Bigg) E_{i,-j} & \text{reversing the column index}\\
        &= \sum_{i',j'=1}^{n+1} {\rm L}(X)_{i', -j'} \Bigg(\sum_{\substack{i \leq i', \\ j \leq j' \phantom{,}}}  E_{i,-j} \Bigg) & \text{interchanging the sums}\\
        &= \sum_{i',j'=1}^{n} X_{i', -j'} \left[\begin{array}{c|c}
        0 & \mathbf{1}_{i' \times j'} \\
        \hline
        0 & 0
    \end{array}\right],
    \end{align*}
    where the last line follows from observing in~\eqref{L} that ${\rm L}(X)_{i',-j'} = X_{i', -j'}$ if $1 \leq i', j' \leq n$, and 0 otherwise.
\end{proof}

\begin{remark} 
The simple Monge matrices are fundamental in the following sense.
(The discussion here is just a more detailed treatment following from well-known properties of Monge matrices; see~\cite{Burkard}*{Lemma~2.1} or~\cite{Atallah}*{Lemma~1}, for example.)
Every Monge matrix $A \in \M_n$ can be decomposed uniquely in the form
\begin{equation*}
    \label{Monge decomp}
    A = \underbrace{\Sigma \big( \Delta(A) \big)}_{\text{simple matrix}} \;\; + \underbrace{S(A)}_{\text{sum matrix}},
\end{equation*}
where $S(A) \in \mathbb{Z}^{(n+1) \times (n+1)}$ is determined by the entries along the lower-left boundary of $A$:
\[
S(A)_{ij} = A_{i,1} + A_{n+1, j} - A_{n+1, 1}.
\]
The matrix $S(A)$ is called a \emph{sum matrix} (following~\cite{Burkard}*{p.~100}) because it is the addition table of the two vectors $(A_{1,1}, \ldots, A_{n+1,1}) \in \N^{n+1}$ and $(0, A_{n+1, 2} - A_{n+1,1}, \ldots, A_{n+1, n+1} - A_{n+1,1}) \in \mathbb{Z}^{n+1}$.  
In this sense, the component $S(A)$ contains relatively little information about a Monge matrix $A$.
Indeed, it is the simple part of $A$ that contains most of the information, and which approximates $A$ up to some combination of constant rows and constant columns.
\end{remark}

In applications of Monge matrices, it is common to study the \emph{distance product}.
This is just the standard matrix product, where ordinary addition and multiplication are replaced by the operations
\begin{align*}
    x \oplus y &\coloneqq \min\{x,y\},\\
    x \otimes y &\coloneqq x + y,
\end{align*}
respectively.
We will use the symbol $\odot$ for the distance product.
Thus the product $A \odot B = C$ is defined by
\begin{equation}
\label{min plus matrix product}
    C_{ij} = \bigoplus_k A_{ik} \otimes B_{kj} = \min_k\{A_{ik} + B_{kj}\}.
\end{equation}

It is well known that $\M_n$ is closed under the distance product $\odot$, which is associative; see, for example,~\cite{Atallah}*{\S4}, where the authors use the term ``convex matrix'' rather than ``Monge matrix.''
Therefore $(\M_n, \odot)$ is a semigroup
(but not a monoid, since the identity matrix would require $0$'s on the diagonal and an ``$\infty$'' element elsewhere).
Since $0 \odot 0 = 0$ (where $0$ denotes the zero matrix in appropriate dimensions), it is easy to see (by block matrix multiplication) that the subset $\M^0_n$ of simple Monge matrices is also closed under $\odot$, and so $(\M^0_n, \odot)$ also forms a semigroup.

\subsection{Norm growth series of \texorpdfstring{$\M^0_n$}{the Monge semigroup}}

It is natural to equip matrices in $\M^0_n$ with  the max norm or the $L_{1,1}$ norm, defined as follows:
\begin{align*}
    \| A \|_{\max} &\coloneqq \max_{i,j} A_{ij},\\
    \|A\|_{1,1} &\coloneqq \sum_{i,j} A_{ij}.
\end{align*}
(Note that since each $A \in \M^0_n$ has nonnegative entries, there is no need for absolute values in the definitions above.)
It is straightforward to verify from the definition~\eqref{min plus matrix product} that both norms above are submultiplicative with respect to the distance product, meaning that $\| A \odot B \| \leq \|A\| \otimes \|B\|$.
Since $\otimes$ here denotes ordinary addition, this means that both the max norm and the $L_{1,1}$ norm induce a filtration $\mathscr{F}^0 \subseteq \mathscr{F}^1 \subseteq \mathscr{F}^2 \subseteq \cdots$ on $\M^0_n$, in the sense that
\[
    \mathscr{F}^k \odot \mathscr{F}^\ell \subseteq \mathscr{F}^{k+\ell},
\]
where
\begin{equation} 
    \label{Fk}
    \mathscr{F}^k = \mathscr{F}^k_{\bullet}(n) \coloneqq \Big\{ A \in \M^0_n : \|A\|_\bullet \leq k \Big\},
\end{equation}
with the symbol $\bullet$ as a placeholder for either ``$\max$'' or ``$1,1$.''

In order to understand the growth of these filtered components $\mathscr{F}^k$, it suffices to determine the number of matrices whose norm \emph{equals} a given integer, since then we can recover the growth of the $\mathscr{F}^k$ by taking partial sums.
Hence for each $k \in \N$ we define
\begin{equation}
    \label{Gk}
    \mathscr{G}^k_\bullet(n) \coloneqq \Big\{ A \in \M^0_n : \|A\|_\bullet = k\Big\},
\end{equation}
where again the symbol $\bullet$ is a placeholder for either ``$\max$'' or ``$1,1$.''
(The notation $\mathscr{G}$ is meant to evoke the associated \emph{graded} ring of a filtration.)
The cardinalities of the components $\mathscr{G}^k_\bullet(n)$ can be encoded in a formal power series in the indeterminate $q$, where the coefficient of $q^k$ equals $|\mathscr{G}^k_\bullet(n)|$.
We will call this the \emph{growth series} of $\M^0_n$ with respect to the norm $\| \;\; \|_\bullet$, as written out below in two equivalent forms:
\[
    \sum_{k=0}^\infty \big| \mathscr{G}^k_{\bullet}(n)\big| q^k = \sum_{A \in \M^0_n} q^{\|A\|_\bullet}.
\]
% On the other hand, if (as before) we view $(\M^0_n, \odot)$ as the semigroup under the (min, +) matrix product, then the chain of subsets $(\M^0_n)_{\leq s}$ gives a filtration, since
% \[
% \big( \M^0_n \big)_{\leq s} \odot \big( \M^0_n \big)_{\leq t} \subseteq \big( \M^0_n \big)_{\leq s+t}.
% \]
% To see this, it suffices to deduce from the definition~\eqref{min plus matrix product} that if $A \odot B = C$, then $|A| + |B| \geq |C|$.
% To see this, in turn, for each pair $(i,j)$ choose the unique $k \in \{1, \ldots, n+1\}$ such that $k \equiv {i+j-1 \: ({\rm mod} \: n+1)}$.
% Then the set $\{A_{ik}, B_{kj} : 1 \leq i,j \leq n+1\}$ includes each entry of $A$ and each entry of $B$ exactly once, and so the sum of the elements in this set equals $|A| + |B|$.
% But by the right-hand side of~\eqref{min plus matrix product}, this sum is greater than or equal to the sum of all the $C_{ij}$'s where we take minima over all possible values of $k$.
In the following theorem we give a closed-form generating function for the growth series of $\M^0_n$ with respect to each of the two norms mentioned above.

\begin{theorem}
    \label{thm:growth series}

    Let $\mathscr{G}^k_\bullet(n)$ be as defined in~\eqref{Gk}.
    We have the following growth series of $\M^0_n$:

    \begin{enumerate}
        \item $\displaystyle \sum_{k=0}^\infty \Big| \mathscr{G}^k_{\max}(n) \Big| q^k = \frac{1}{(1-q)^{n^2}}$.

        \item $\displaystyle \sum_{k=0}^\infty \Big| \mathscr{G}^k_{1,1}(n) \Big| q^k = \prod_{i,j=1}^n \frac{1}{1-q^{ij}}$.
    \end{enumerate}
    
\end{theorem}

\begin{proof} In both cases, we will exploit the bijection $\Phi : \N^{n \times n} \longrightarrow \M^0_n$ in~\eqref{dist density bijection}, in order to pull Monge matrices back to their preimages in $\N^{n \times n}$, where calculations are more straightforward.

\begin{enumerate}

\item Let $X \in \N^{n \times n}$.
By~\eqref{A Sigma} and~\eqref{dist density bijection}, the maximum entry in $\Phi(X)$ is the upper-right entry, which is the sum of all entries in $X$.
Hence $\| \Phi(X) \|_{\max} = \sum_{i,j} X_{ij}$.
We therefore have
\begin{align*}
    \sum_{k=0}^\infty \Big| \mathscr{G}^k_{\max}(n) \Big| q^k &= \sum_{A \in \M^0_n} q^{\|A\|_{\max}}\\
    &= \sum_{X \in \N^{n \times n}} q^{\| \Phi(X) \|_{\max}} & \text{by~\eqref{dist density bijection}}\\
    &= \sum_{X \in \N^{n \times n}} q^{\sum_{i,j} X_{ij}}\\
    &= \sum_{X \in \N^{n \times n}} \left(\prod_{i,j=1}^n q^{X_{ij}}\right)\\
    &= \prod_{i,j=1}^n \left( \sum_{\ell \in \N} q^\ell \right) \\
    &= \prod_{i,j=1}^n \frac{1}{1-q} \\
    &= \frac{1}{(1-q)^{n^2}}.
\end{align*}

\item We have

    \begin{align*}
        \label{first line in proof}
        \sum_{k=0}^\infty \Big| \mathscr{G}^k_{1,1}(n) \Big| q^k &= \sum_{A \in \M^0_n} q^{\|A\|_{1,1}} \\
        &= \sum_{X \in \N^{n \times n}} q^{\|\Phi(X)\|_{1,1}} & \text{by~\eqref{dist density bijection}}\\
        &= \sum_{X \in \N^{n \times n}} q^{\left\|\sum_{i,j} X_{i,-j} \scalebox{.5}{$\left[\begin{array}{c|c}
        0 & \mathbf{1}_{i \times j} \\
        \hline
        0 & 0
    \end{array}\right]$} \right\|_{1,1}} & \text{by Lemma~\ref{lemma:Phi}}\\
    &= \sum_{X \in \N^{n \times n}} q^{\sum_{i,j} ij \cdot X_{i,-j}} & \text{since $\left\|\raisebox{2pt}{\scalebox{.5}{$\left[\begin{array}{c|c}
        0 & \mathbf{1}_{i \times j} \\
        \hline
        0 & 0
    \end{array}\right]$}} \right\|_{1,1} = ij$}\\
    &= \sum_{X \in \N^{n \times n}} \Bigg(\prod_{i,j=1}^n (q^{ij})^{X_{i,-j}}\Bigg)\\
    &= \prod_{i,j=1}^n \left(\sum_{\ell \in \N} (q^{ij})^{\ell} \right)\\
    &= \prod_{i,j=1}^n \frac{1}{1-q^{ij}}. &\qedhere
    \end{align*}

    \end{enumerate}
\end{proof}

\begin{cor}
    \label{cor:growth series}
    Let $\mathscr{F}^k_\bullet (n)$ be as defined in~\eqref{Fk}.
    We have the following generating functions:

    \begin{enumerate}
        \item $\displaystyle\sum_{k=0}^\infty \Big| \mathscr{F}^k_{\max}(n) \Big| q^k = \frac{1}{(1-q)^{n^2 +1}}$.

        \item $\displaystyle\sum_{k=0}^\infty \Big| \mathscr{F}^k_{1,1}(n) \Big| q^k = \frac{1}{1-q}\prod_{i,j=1}^n \frac{1}{1-q^{ij}}$.
    \end{enumerate}
\end{cor}

\begin{proof}
    By the definition~\eqref{Fk}, it is clear that $|\mathscr{F}^k_\bullet(n)| = \sum_{i=0}^k |\mathscr{G}^i_\bullet(n)|$.
    Hence the generating functions in this corollary are obtained by taking the partial sums of the coefficients in Theorem~\ref{thm:growth series}.
    It is a standard fact that the generating function for partial sums is obtained by multiplying the original generating function by $\frac{1}{1-q}$.
\end{proof}

\subsection{Combinatorial interpretations}

The generating functions in part (1) of Theorem~\ref{thm:growth series} and Corollary~\ref{cor:growth series} make it immediately clear that the max norm cardinalities in $\M^0_n$, along with their partial sums, are given by binomial coeffients:
\[
\Big| \mathscr{G}^k_{\max}(n) \Big| = \binom{n^2 + k -1}{k}, \qquad \Big| \mathscr{F}^k_{\max}(n) \Big| = \binom{n^2 + k}{k}.
\]
Moreover, there is an elegant way to interpret the elements of $\mathscr{F}^k_{\max}(n)$ in terms of $(n,n,k)$-boxed plane partitions.
Recall that in general, an \emph{$(a,b,c)$-boxed plane partition} is a Young tableau with at most $a$ rows and at most $b$ columns, where the entries are taken from the set $\{1, \ldots, c\}$ and weakly decrease along rows and columns.
Let
\[
    \mathscr{P}(a,b,c) \coloneqq \Big\{ \text{$(a,b,c)$-boxed plane partitions}\Big\}.
\]
A plane partition can also be depicted as a three-dimensional arrangement of unit cubes, where the entry in each box of the plane partition is the number of cubes stacked on that box; via this depiction, an $(a,b,c)$-boxed plane partition is one that fits inside a rectangular prism of length $a$, width $b$, and height $c$.
The example below shows the same plane partition in its two guises described above, namely the two-dimensional (tableau) and three-dimensional (unit cube) depiction:
\begin{equation}
\label{PP pictures}
\ytableausetup{centertableaux}
\ytableaushort{8521,531,32,11} \qquad \input{planepartition}
\end{equation}

In the three-dimensional depiction, each row of a plane partition becomes a \emph{wall} running from west to east (and weakly decreasing in height as it does).
To an observer standing to the south, each wall partially blocks the wall behind it (i.e., the wall immediately to its north, i.e., the row directly above it in the two-dimensional diagram).
For each wall, we say that its \emph{southern face} is the collection of squares that are visible to this observer, i.e., the squares that are not blocked by the wall in front of it.
In~\eqref{PP pictures}, we have shaded the southern faces of each of the four walls.
We say the southern face of a wall is \emph{weakly decreasing} if the heights of its columns form a weakly decreasing sequence from west to east.
In~\eqref{PP pictures}, all four southern faces are weakly decreasing; in particular, the four sequences of heights are
\[
(3,2,1,1), \quad (2,1,1), \quad (2,1), \quad (1,1).
\]

\begin{prop}
\label{prop:plane partitions}
    Let $\mathscr{F}^k_{\max}(n)$ be as defined in~\eqref{Fk}.
    The number of matrices in $\mathscr{F}^k_{\max}(n)$, namely $\binom{n^2 + k}{k}$, equals the number of $(n,n,k)$-boxed plane partitions whose southern faces are all weakly decreasing.
\end{prop}

\begin{proof}
    We give a bijective proof.
    Note that each plane partition $\pi \in \mathscr{P}(n,n,k)$ can be viewed as a matrix in $\N^{n \times n}$ with entries in $\{0, \ldots, k\}$, where the entries of $\pi$ are the matrix entries in the corresponding positions (with zeros elsewhere).
    Let $\psi: \mathscr{P}(n,n,k) \longrightarrow \N^{(n+1) \times (n+1)}$ be the injective map that takes a plane partition, reflects it horizontally, and then pads it with zeros along the lower-left boundary (via the map ${\rm L}$ in~\eqref{L}).
    Then we have $\psi(\mathscr{P}(n,n,k)) \supseteq \mathscr{F}^k_{\max}(n)$, since by Lemma~\ref{lemma:Phi} every simple Monge matrix is the horizontal reflection of some plane partition.

    It remains to show that $\psi(\pi) \in \M^0_n$ if and only if the southern faces of $\pi$ are all weakly decreasing.
    Let $\pi_{ij}$ denote the entry in the $j$th box of the $i$th row of $\pi$ (in its standard depiction as a two-dimensional tableau), where $\pi_{ij} = 0$ if $\pi$ does not contain a box in position $(i,j)$.
    The columns in the southern face of the $i$th wall in $\pi$ are simply the differences $\pi_{ij}-\pi_{i+1,j}$, for each $1 \leq j \leq n$.
    Therefore, the southern faces of $\pi$ are all weakly decreasing if and only if, for all $1 \leq i,j \leq n$, we have
    \[
    \pi_{ij} - \pi_{i+1,j} \geq \pi_{i,j+1} - \pi_{i+1, j+1}.
    \]
    Due to horizontal reflection, this condition is equivalent to the following condition on the matrix $A \coloneqq \psi(\pi)$:
    \[
    A_{i,j+1} - A_{i+1,j+1} \geq A_{ij} - A_{i+1,j},
    \]
    which in turn is equivalent to
    \[
    A_{ij} + A_{i+1, j+1} \leq A_{i,j+1} + A_{i+1, j},
    \]
    which by~\eqref{Monge condition contiguous} is equivalent to $A$ having the Monge property.
\end{proof}

\begin{example}
    Let $n \geq 4$ and $k \geq 8$.
    Then the plane partition~\eqref{PP pictures} is an element of $\mathscr{P}(n,n,k)$, and we have already observed above that it is weakly decreasing.
    Moreover, its image under $\psi$ is the matrix with the upper-right $4 \times 4$ block
    \[
    \begin{bmatrix}
        1 & 2 & 5 & 8\\
        0 & 1 & 3 & 5\\
        0 & 0 & 2 & 3\\
        0 & 0 & 1 & 1
    \end{bmatrix}
    \]
    (and zeros elsewhere).
    Note that this matrix has the Monge property, and thus lies in $\mathscr{F}^k_{\max}(n)$.
\end{example}

We now turn from the max norm cardinalities to the $L_{1,1}$ norm cardinalities, which are somewhat more interesting,
especially if we consider simple Monge matrices of arbitrary dimension (by letting $n \rightarrow \infty$).
In fact, it turns out that these $L_{1,1}$ norm cardinalities arise in several interesting combinatorial settings, seemingly far removed from Monge matrices.

Given two matrices $A$ and $B$ with possibly different dimensions, consider the equivalence relation whereby $A \sim B$ if $A = \left[\begin{smallmatrix} 0 & B\\ 0 & 0 \end{smallmatrix} \right]$ or $B = \left[\begin{smallmatrix} 0 & A\\ 0 & 0 \end{smallmatrix} \right]$.
Then we can define the set of all simple Monge matrices of arbitrary dimension:
\[
    \M^0_\infty \coloneqq \left(\bigcup_{n=0}^\infty \M^0_n \right) \Bigg/ \sim.
\]
If $A \in \M^0_n$ for some $n$, then let $\overline{A} \in \M^0_\infty$ denote the equivalence class of $A$.
Note that $\| \overline{A} \|_{1,1} = \|A\|_{1,1}$ is well-defined, since every matrix in the equivalence class of $A$ has the same nonzero entries as $A$.
Therefore, it makes sense to define the set
    \begin{equation}
        \label{M0 infinity}
        \mathscr{G}^k_{1,1}(\infty) \coloneqq \Big\{ \overline{A} \in \M^0_\infty : \| \overline{A} \|_{1,1} = k \Big\}.
    \end{equation}
    Note that for fixed $k \in \N$, the set $\mathscr{G}^k_{1,1}(n)$ eventually stabilizes (modulo $\sim$) exactly when $n \geq k$:
    \[
        \mathscr{G}^k_{1,1}(k-1) \subsetneq \mathscr{G}^k_{1,1}(k) = \mathscr{G}^k_{1,1}(k+\ell), \qquad \text{for all }\ell \in \N.
    \]
    This is because the matrix $A$ (or its antitranspose) consisting of a $1 \times k$ block of 1's in the upper-right, with zeros elsewhere, is such that $\|A\|_{1,1} = k$, and $A$ occurs in $\M^0_{n}$ if and only if $n \geq k$.
    The upshot is that for each $k \in \N$, the set
    \begin{equation}
    \label{stable}
        \mathscr{G}^k_{1,1}(\infty) = \Big\{ \overline{A} : A \in \mathscr{G}^k_{1,1}(k) \Big\}
    \end{equation}
    is indeed finite.
    Moreover, the generating function of the $L_{1,1}$ growth series for $\M^0_\infty$ is obtained by letting $n \rightarrow \infty$ in Theorem~\ref{thm:growth series}(2):
    \begin{align}
        \label{gf}
        \begin{split}
        \sum_{k=0}^\infty \Big| \mathscr{G}^k_{1,1}(\infty) \Big| q^k &= \prod_{i,j=1}^\infty \frac{1}{1-q^{ij}}\\
        &= 1 + q + 3 q^2 + 5 q^3 + 11 q^4 + 17 q^5 + 34 q^6 + 52 q^7 + 94 q^8 + \cdots .
        \end{split}
    \end{align}
    This sequence $(|\mathscr{G}^k_{1,1}(\infty)|)_{k=0}^\infty= (1,1,3,5,11,17,34,52,94,145,244,370,603,\ldots)$ arises in several surprising combinatorial settings.
    We record these alternative interpretations in the following proposition, along with explicit bijective proofs.
    Immediately following the proposition, we give Example~\ref{ex:bijections} in order to illustrate these bijections.

    \begin{prop}
        \label{prop:interpretations}
        Let $\mathscr{G}^k_{1,1}(\infty)$ be as defined in~\eqref{M0 infinity}.
        The number of elements in $\mathscr{G}^k_{1,1}(\infty)$ also equals
        \begin{enumerate}
        \item the number of  partitions of $k$, when there are $d(a)$ different copies of each part $a = 1, 2, \ldots$, where $d(a)$ is the number of divisors of $a$.
        
        \item the number of factorization patterns of $k$, in the sense of Hultquist--Mullen--Niederreiter~\cite{Hultquist}.

        \item the number of partitions of $k$ where there are unlimited distinguishable but unlabeled parts of each size.

        \item the number of unital *-subalgebras of the $k \times k$ complex matrices, up to unitary similarity. 
        (A unital *-subalgebra is a subspace that is closed under multiplication and conjugate transpose, and which contains the identity matrix.)

        % \item the number of partitions of $s$ having parts consisting of runs of equal parts.

    \end{enumerate}
        
    \end{prop}

    \begin{proof}

    Let $\overline{A} \in \M^0_\infty$.
    By~\eqref{dist density bijection} and Lemma~\ref{lemma:Phi}, we have
    \[
    \| \overline{A} \|_{1,1} = \| A \|_{1,1} = \sum_{i,j} ij \cdot \Delta(A)_{i, -j},
    \]
    where again we write $-j$ to denote the $j$th column counting from the right.
    Hence $\|\overline{A}\|_{1,1}$ is a weighted sum of the (finitely many) nonzero entries of the density matrix $\Delta(A)$.
    Therefore, by the bijections in~\eqref{dist density bijection}, and by~\eqref{stable}, we have a bijection
    \begin{equation}
        \label{count}
        \mathscr{D}_k \coloneqq \left\{M \in \N^{k \times k} : \sum_{i,j=1}^k ij \cdot M_{i,-j} = k \right\} \xlongrightarrow{\overline \Sigma} \mathscr{G}^k_{1,1}(\infty),
    \end{equation}
    where $\overline{\Sigma}$ is defined exactly as  $\Sigma$ is defined in~\eqref{A Sigma}, but with the dimension $n+1$ replaced by $k$, and its image sent to its equivalence class modulo $\sim$.
    Hence in this proof it will suffice to construct bijections between the objects in (1)--(4) and the set~$\mathscr{D}_k$ of density matrices on the left-hand side of~\eqref{count}.

    \begin{enumerate}
        \item Let $\mathscr{P}_k$ denote the set of partitions $\pi \vdash k$, where the parts can be chosen from the set
        \[
        \bigcup_{a=1}^\infty \left\{a_{[1]}, \ldots, a_{[d(a)]} \right\},
        \]
        with the bracketed subscripts designating the $d(a)$ different copies of each part $a$.
        Let $d_i(a)$ denote the number of divisors of $a$ which are $\leq i$.
        For example, we have $d_4(20) = 3$ because there are three divisors of $20$ which are $\leq 4$ (namely $1$, $2$, and $4$).
        Clearly $d_i(a) \leq d(a)$ for all $i$.
        % Let $\mathbf{d}(a)$ denote the list of divisors of $a$, written in increasing order, and let $\mathbf{d}_i(a)$ denote the $i$th divisor in this list; for example, $\mathbf{d}_3(8) = 4$ because $\mathbf{d}(8) = (1,2,4,8)$.

        We now exhibit a bijection $\mathscr{P}_k \longleftrightarrow \mathscr{D}_k$.
        To each partition $\pi \in \mathscr{P}_k$, we associate the matrix $M$ in which each entry $M_{i,-j}$ equals the number of occurrences of the part $(ij)_{[d_i(ij)]}$ in $\pi$.
        This map is one-to-one, and the fact that $\pi \vdash k$ guarantees that $M \in \mathscr{D}_k$.
        In the opposite direction, to each matrix $M \in \mathscr{D}_k$, we associate the partition $\pi$ which contains exactly $M_{i,-j}$ many occurrences of the part $(ij)_{[d_i(ij)]}$, for all $1 \leq i,j \leq k$.
        Again, this map is one-to-one, and the defining condition on $\mathscr{D}_k$ in~\eqref{count} guarantees that $\pi \in \mathscr{P}_k$.
        Therefore we have a bijection $\mathscr{P}_k \longleftrightarrow \mathscr{D}_k$.

        \item We refer the reader to~\cite{AM}*{pp.~121--2}, where a \emph{factorization pattern} is defined (first introduced in~\cite{Hultquist}), and where an explicit bijection is given between $\mathscr{P}_k$ and the set of factorization patterns of $k$.
        Composing this with the bijection in part (1) then produces the desired bijection with $\mathscr{D}_k$.

        \item Let $\mathscr{R}_k$ denote the set of partitions $\rho \vdash k$, where there are unlimited distinguishable but unlabeled parts of each size.
        This means that for each part $a$ occurring in $\rho$, we decorate the occurrences of $a$ with subscripts $(1), (2), \ldots$, such that the number of $a_{(j)}$'s is greater than or equal to the number of $a_{(j+1)}$'s for all $j$ (this is the ``unlabeled'' condition).

        We now exhibit a bijection $\mathscr{R}_k \longleftrightarrow \mathscr{D}_k$.
        To each $\rho \in \mathscr{R}_k$, we associate the matrix $M$, with entries defined via
        \[
        M_{i,-j} = \Big(\text{number of $i_{(j)}$'s in $\rho$}\Big) - \Big(\text{number of $i_{(j+1)}$'s in $\rho$}\Big).
        \]
        This map is one-to-one, and the fact that $\rho \vdash k$ guarantees that $M \in \mathscr{D}_k$.
        In the opposite direction, to a matrix $M \in \mathscr{D}_k$, we associate the partition $\rho$ which contains exactly $\sum_{\ell \geq j} M_{i,-\ell}$ many occurrences of the part $i_{(j)}$.
        This map is also one-to-one, and the fact that $M \in \mathscr{D}_k$ guarantees that $\rho \vdash k$.
        Hence we have a bijection $\mathscr{R}_k \longleftrightarrow \mathscr{D}_k$.
    
        \item We refer the reader to the exposition written by Nathaniel Johnston for OEIS entry {\href{https://oeis.org/A215905}{A215905}}.
        The key fact, for our purposes, is that every *-subalgebra of the $k \times k$ complex matrices ${\rm M}_k \coloneqq {\rm M}_k(\mathbb{C})$ is unitarily similar to a direct sum
        of full matrix algebras of smaller dimension, where summands of the same dimension may or may not be forced to be identical (and if they are so forced, then we obtain the diagonal subalgebra of their direct sum).
        The unital condition forces the sums of the squares of the dimensions of the summands to equal $k^2$.
        Hence, the set of equivalence classes of unital *-subalgebras of ${\rm M}_k$ is parametrized by the set $\mathscr{R}_k$ from part (3) above, namely, the set of partitions of $k$ which allow for distinguishable but unlabeled parts.
        For example, if $k=20$, then the partition
        \[
        \rho = \left(6_{(1)}, 3_{(1)}, 2_{(1)}, 2_{(1)}, 2_{(2)}, 1_{(1)},1_{(1)}, 1_{(1)}, 1_{(2)},1_{(3)} \right)
        \]
        corresponds to the equivalence class of the unital *-subalgebra
        \[
        {\rm M}_6 \oplus {\rm M}_3 \oplus \underbrace{\Delta({\rm M}_2 \oplus {\rm M}_2)}_{\cong {\rm M}_2} \oplus \: {\rm M}_2 \oplus \underbrace{\Delta({\rm M}_1 \oplus {\rm M}_1 \oplus {\rm M}_1)}_{\cong {\rm M}_1} \oplus \: {\rm M}_1 \oplus {\rm M}_1,
        \]
        where $\Delta$ denotes the diagonal subalgebra consisting of tuples of identical elements. \qedhere   
    \end{enumerate}

    \end{proof}

    \begin{example}
        \label{ex:bijections}
        We give an example illustrating the bijections in the proofs of Proposition~\ref{prop:interpretations}.
        Since parts (2) and (4) are straightforward modifications of (1) and (3), respectively, it suffices to give an example of the following bijections:
        
        \input{Bijection_diagram}

        \noindent For our example, we take $k=77$, and we start with the density matrix $M \in \mathscr{D}_{77}$ whose entries are all zero except for its $4 \times 4$ upper-right block
        \[
        \begin{bmatrix}
            2 & 0 & 0 & 1\\
            0 & 1 & 1 & 2\\
            1 & 1 & 3 & 1\\
            0 & 0 & 0 & 3
        \end{bmatrix}.
        \]
        One can see that indeed $M \in \mathscr{D}_{77}$ because the $L_{1,1}$ norm of its image $\overline{\Sigma}(M)$ is $77$:
        \[
        \overline{\Sigma}(M) = \overline{\begin{bmatrix}
            3 & 5 & 9 & 16\\
            1 & 3 & 7 & 13\\
            1 & 2 & 5 & 9\\
            0 & 0 & 0 & 3
        \end{bmatrix}} \in \mathscr{G}_{1,1}^{77}(\infty).
        \]
        Using the bijective proof of part (1) of Proposition~\ref{prop:interpretations}, we have the partition $\pi \in \mathscr{P}_{77}$ given by
        \[
        \pi = \Big(12_{[3]}, 9_{[2]}, 6_{[2]}, 6_{[3]}^3, 4_{[1]}^2, 4_{[2]}, 4_{[3]}^3, 3_{[2]}, 2_{[2]}^2, 1_{[1]} \Big),
        \]
        where exponents denote multiple occurrences of a part.
        For example, $\pi$ contains three occurrences of $4_{[3]}$ because of the entry $M_{4,-1} = 3$: in particular, the product $4 \cdot 1 = 4$ determines the part $4$, the fact $d_4(4)=3$ determines the copy given by the subscript $[3]$, and the entry $3$ gives the number of occurrences of $4_{[3]}$ in $\pi$.
        Likewise, using the bijective proof of part (3) of Proposition~\ref{prop:interpretations}, we have the partition $\rho \in \mathscr{R}_{77}$ given by
        \[
        \rho = \Big( 4_{(1)}^3, 3_{(1)}^6, 3_{(2)}^5, 3_{(3)}^2, 3_{(4)}, 2_{(1)}^4, 2_{(2)}^2, 2_{(3)}, 1_{(1)}^3, 1_{(2)}^2, 1_{(3)}^2, 1_{(4)}^2 \Big).
        \]
        For example, $\rho$ contains five occurrences of $3_{(2)}$ because the sum of the entries in the $3$rd row of $M$, lying weakly left of the $2$nd column from the right, is five.
    \end{example}

    \begin{remark}
        All four interpretations in Proposition~\ref{prop:interpretations} are listed in OEIS entry~{\href{https://oeis.org/A006171}{A006171}}.
        Interpretation (1), along with several variants, was studied by Agarwal--Mullen~\cite{AM}, who used an equivalent form
        \[
        \prod_{a=1}^\infty \frac{1}{(1-q^a)^{d(a)}}
        \]
        of our generating function~\eqref{gf} to study its asymptotic behavior.
        In interpretation (4), if one relaxes the unital condition (see OEIS entry {\href{https://oeis.org/A215925}{A215925}}), then the number of *-subalgebras of ${\rm M}_k$ equals the cardinality of $\mathscr{F}^k_{1,1}(\infty) \coloneqq \{ \overline{A} \in \M^0_\infty : \| \overline{A} \|_{1,1} \leq k \}$.
    \end{remark}

\section{Main result\texorpdfstring{: $(\N^{n \times n}, \star) \cong (\M^0_n, \odot)$}{}}
\label{sec:main}

Our main result unifies Sections~\ref{sec:Demazure} and~\ref{sec:Monge}, by stating that our extension of the Demazure product is essentially nothing other than the  distance product of simple Monge matrices.
In other words, the procedure of stacking and fusing two kelp beds (Definition~\ref{def:Demazure extended}) actually computes the distance product of the two corresponding simple Monge matrices.
We state the main theorem below along with the running example from earlier in the paper; we defer the proof of the theorem to the end of the section, after we have proved two key lemmas.

\begin{theorem}
    \label{thm:main result}
    The map $\Phi : \N^{n \times n} \longrightarrow \M^0_n$ defined in~\eqref{dist density bijection} is in fact an isomorphism of semigroups $(\N^{n \times n}, \star) \cong (\M^0_n, \odot)$.
\end{theorem}

\begin{example}
    We return once more to our running example from both~\eqref{ex:example from intro} and Example~\ref{ex:example from intro}.
    Recall that we took $n=4$, and we had
    \[
    X = \begin{bmatrix}
        0 & 0 & 0 & 0\\
        1 & 0 & 1 & 0\\
        0 & 0 & 3 & 0\\
        0 & 1 & 0 & 0
    \end{bmatrix},
    \qquad
    Y = \begin{bmatrix}
        0 & 0 & 1 & 1\\
        0 & 0 & 0 & 1\\
        1 & 0 & 0 & 0\\
        0 & 0 & 0 & 0
    \end{bmatrix},
    \qquad
    X \star Y = \begin{bmatrix}
        0 & 0 & 0 & 0\\
        0 & 0 & 1 & 0\\
        0 & 0 & 0 & 1\\
        1 & 0 & 0 & 0
    \end{bmatrix}.
    \]
    Recall from~\eqref{dist density bijection} that $\Phi = \Sigma \circ {\rm L}$ first pads a matrix with 0's around the southwest border, then replaces each entry with the sum of the entries weakly southwest of it.
    (Stated this way, the two operations actually commute.)
    Hence the images of the three matrices shown above are the simple Monge matrices
    \[
    \Phi(X) = \begin{bmatrix}
        0 & 1 & 2 & 6 & 6 \\
        0 & 1 & 2 & 6 & 6 \\
        0 & 0 & 1 & 4 & 4 \\
        0 & 0 & 1 & 1 & 1 \\
        0 & 0 & 0 & 0 & 0        
    \end{bmatrix},
    \qquad
    \Phi(Y) = \begin{bmatrix}
         0 & 1 & 1 & 2 & 4 \\
         0 & 1 & 1 & 1 & 2 \\
         0 & 1 & 1 & 1 & 1 \\
         0 & 0 & 0 & 0 & 0 \\
         0 & 0 & 0 & 0 & 0 \\
    \end{bmatrix},
    \qquad
    \Phi(X \star Y) = \begin{bmatrix}
     0 & 1 & 1 & 2 & 3 \\
     0 & 1 & 1 & 2 & 3 \\
     0 & 1 & 1 & 1 & 2 \\
     0 & 1 & 1 & 1 & 1 \\
     0 & 0 & 0 & 0 & 0 \\
    \end{bmatrix}.
    \]
    To verify Theorem~\ref{thm:main result}, we use~\eqref{min plus matrix product} to take the distance product $\Phi(X) \odot \Phi(Y)$ directly, and we find that indeed
    \[
    \Phi(X) \odot \Phi(Y) = \begin{bmatrix}
        0 & 1 & 2 & 6 & 6 \\
        0 & 1 & 2 & 6 & 6 \\
        0 & 0 & 1 & 4 & 4 \\
        0 & 0 & 1 & 1 & 1 \\
        0 & 0 & 0 & 0 & 0        
    \end{bmatrix}
    \odot
    \begin{bmatrix}
         0 & 1 & 1 & 2 & 4 \\
         0 & 1 & 1 & 1 & 2 \\
         0 & 1 & 1 & 1 & 1 \\
         0 & 0 & 0 & 0 & 0 \\
         0 & 0 & 0 & 0 & 0 \\
    \end{bmatrix} = 
    \begin{bmatrix}
     0 & 1 & 1 & 2 & 3 \\
     0 & 1 & 1 & 2 & 3 \\
     0 & 1 & 1 & 1 & 2 \\
     0 & 1 & 1 & 1 & 1 \\
     0 & 0 & 0 & 0 & 0 \\
    \end{bmatrix}
    = \Phi(X \star Y).
    \]
\end{example}

En route to proving Theorem~\ref{thm:main result}, we first prove two lemmas below.
In their proofs, it will be convenient to restrict a kelp bed $X$ to the subgraph induced by a certain interval of top or bottom vertices.
Given $S,T \subseteq [n]$, we will adopt the standard notation for the subgraph of $X$ induced by $S \times T$, writing
\begin{equation}
    \label{subgraph notation}
    X[S \times T] \coloneqq \Big\{ \underbrace{(i,j), \ldots, (i,j)}_{X_{ij} \text{ copies}} : (i,j) \in S \times T \Big\}.
\end{equation}

\begin{lemma}
    \label{lemma:Sigma X star Y}
    Let $\Phi$ be the map defined in~\eqref{dist density bijection}, and let $\star$ denote the Demazure product from Definition~\ref{def:Demazure extended}.
    For all $1 \leq i,j \leq n+1$, we have
    \[
    \Phi(X \star Y)_{ij} = {\rm wt}\Big(X\big[\{i, \ldots, n\} \times [n]\big], \; Y \big[[n] \times \{1, \ldots, j-1\}\big]\Big)
    \]
    with the weight $({\rm wt})$ as defined in Definition~\ref{def:tangledness}.
\end{lemma}

\begin{proof}
    From step (2) of Definition~\ref{def:Demazure extended}, it is clear that for each $1 \leq \ell \leq w$, we have ${\rm wt}(X_\ell,Y_\ell) = \ell$, since ${\rm wt}(X_0, Y_0) = 0$ and the weight increases by exactly 1 every time we add a pair $(x_\ell, y_\ell)$.
    Moreover, since we successively choose the kelps $x_\ell$ from right to left (with respect to their top vertex), and the corresponding kelps $y_\ell$ from left to right (with respect to their bottom vertex), we have 
    \begin{equation}
        \label{number equals weight}
        \# (X \star Y)[\{i, \ldots, n\} \times \{1, \ldots, j\}] = {\rm wt}\Big(X\big[\{i, \ldots, n\} \times [n]\big], \: Y\big[[n] \times \{1, \ldots, j\}\big] \Big).
    \end{equation}
    Meanwhile, rewriting the definitions~\eqref{A Sigma},~\eqref{L}, and~\eqref{dist density bijection} in terms of our new induced subgraph notation, we have the general fact
    \begin{equation}
        \label{fact Phi number}
        \Phi(X)_{ij} = \#X[\{i, \ldots, n\} \times \{1, \ldots, j-1\}].
    \end{equation}
    Combining this fact with~\eqref{number equals weight}, we obtain
    \begin{align*}
        \Phi(X \star Y)_{ij} &= \#(X \star Y)[\{i, \ldots, n\} \times \{1, \ldots, j-1\}] \\
        &= {\rm wt}\Big(X\big[\{i, \ldots, n\} \times [n]\big], \: Y\big[[n] \times \{1, \ldots, j-1\}\big] \Big). \qedhere
    \end{align*}
\end{proof}

\begin{lemma}
    \label{lemma:tangled and min plus}
    Let $X,Y \in \N^{n \times n}$.
    We have
    \begin{equation}
    \label{min in lemma}
        {\rm wt}(X,Y) = \min_{1 \leq k \leq n+1} \Big\{ \#X\big[[n] \times \{1, \ldots, k-1 \}\big] + \# Y \big[\{k, \ldots, n\} \times [n]\big] \Big\}.
    \end{equation}
\end{lemma}

\begin{proof}
    Let $w$ be the weight of $(X,Y)$.
    Then by Definition~\ref{def:tangledness}, there exist up--down pairs $(x_1, y_1), \ldots, (x_w, y_w) \in X \times Y$, in which no single copy of a kelp appears more than once.
    By Definition~\ref{def:up-down}, for all $1 \leq \ell \leq w$, we have $x_\ell = (a_\ell, b_\ell)$ and $y_\ell = (c_\ell, d_\ell)$ with $b_\ell \leq c_\ell$.
    Thus for all $1 \leq k \leq n+1$, and all $1 \leq \ell \leq w$, we have $b_\ell < k$ or $c_\ell \geq k$, or both.
    Thus the right-hand side of ~\eqref{min in lemma} is at least $w$.
    % We will now exhibit a specific choice of up--down pairs $(x_i, y_i)$ such that the quantity~\eqref{min in lemma} is exactly $w$.
    We must now show that it is exactly $w$.
    
    To this end, in choosing the up--down pairs $(x_\ell,y_\ell)$, choose the $x_\ell$ so that $b_1 \leq \cdots \leq b_w$; then choose the $y_\ell$ so that $c_\ell$ is as small as possible while satisfying $c_\ell \geq b_\ell$.
    Let $y' = (c',d') \in Y \setminus \{y_1, \ldots, y_w\}$ such that $c'$ is as large as possible; if this set is empty, then put $c'=0$.
    We claim that the right-hand side of~\eqref{min in lemma} equals $w$ at the index $k = c'+1$.

    Taking $k = c'+1$ in~\eqref{min in lemma}, we first observe that for each $1 \leq i \leq w$, we have either $b_\ell < k$ or $c_\ell \geq k$, but not both.
    This is because if both were true, then we would have $b_\ell \leq c' < c_\ell$, which violates the minimality of $c_\ell$ (and if $c'=0$, then we would have $b_\ell = 0$, which is impossible since $1 \leq b_\ell \leq n$).
    Hence the subset $\{x_1, y_1, \ldots, x_w, y_w\}$ contributes exactly 1 to the right-hand side of~\eqref{min in lemma}.
    We must show that each kelp in $X \setminus \{x_1, \ldots, x_w\}$ and in $Y \setminus \{y_1, \ldots, y_w\}$ contributes 0.
    This is immediate for $Y \setminus \{y_1, \ldots, y_w\}$, since every element is of the form $(b,c)$ with $c \leq c'$, and thus does not contribute to the right-hand side of~\eqref{min in lemma} at $k = c'+1$.
    Finally, suppose (toward contradiction) that $x = (a,b) \in X \setminus \{x_1, \ldots, x_w\}$ and $b \leq c'$.
    Then $\{(x_\ell,y_\ell) : 1 \leq \ell \leq w \} \cup \{(x,y')\}$ is a system of up--down pairs for $(X,Y)$, which is a contradiction since it implies that ${\rm wt}(X,Y) > w$.
\end{proof}

\begin{proof}[Proof of Theorem~\ref{thm:main result}]

Since $\Phi$ is a bijection of sets, we need only show that $\Phi(X \star Y) = \Phi(X) \odot \Phi(Y)$.
Let $1 \leq i,j \leq n+1$.
We have
\begin{align*}
    \Phi(X \star Y)_{ij} &= {\rm wt}\Big(X\big[\{i, \ldots, n\} \times [n]\big], Y \big[[n] \times \{1, \ldots, j-1\}\big]\Big) & \text{by Lemma~\ref{lemma:Sigma X star Y}} \\
    &= \min_{1 \leq k \leq n+1} \Big\{ \#X\big[\{i, \ldots, n\} \times \{1, \ldots, k-1 \}\big] + \# Y \big[\{k, \ldots, n\} \times \{1, \ldots, j-1\}\big] \Big\} & \text{by Lemma~\ref{lemma:tangled and min plus}} \\
    &= \min_{1 \leq k \leq n+1} \Big\{ \Phi(X)_{ik} + \Phi(Y)_{kj} \Big\} & \text{by~\eqref{fact Phi number}}\\
    &= \big(\Phi(X) \odot \Phi(Y) \big)_{ij} & \text{by~\eqref{min plus matrix product}},
\end{align*}
and the result follows.
\end{proof}

\bibliographystyle{amsplain}
\bibliography{references}

\end{document}

%% file: Intro_example.tex
\begin{tikzpicture}[scale=.5,every node/.style={scale=.75}]

% Number of nodes
\def\n{4}

% Draw the first row of nodes
\foreach \i in {1,...,\n} {
    \node[draw, fill=black, circle, inner sep=1.5pt] (a\i) at (\i, 1) {};
    \node[above=2pt] at (a\i) {};
}

% Draw the second row of nodes
\foreach \i in {1,...,\n} {
    \node[draw, fill=black, circle, inner sep=1.5pt] (b\i) at (\i, 0) {};
    \node[below=2pt] at (b\i) {};
}

\draw[thick] (a2) .. controls +(0,-0.5) and +(0,0.5) .. (b1);
\draw[thick] (a2) .. controls +(0,-0.5) and +(-.5,0) .. (b3);
\draw[thick] (a3) .. controls +(0,-0.5) and +(0,0.5) .. (b3);
\draw[thick] (a3) .. controls +(.25,-0.5) and +(.25,0.5) .. (b3);
\draw[thick] (a3) .. controls +(-.25,-0.5) and +(-.25,0.5) .. (b3);
\draw[thick] (a4) .. controls +(0,-0.5) and +(0,0.5) .. (b2);

\end{tikzpicture}
\hspace{45pt}
\begin{tikzpicture}[scale=.5,every node/.style={scale=.75}]

% Number of nodes
\def\n{4}

% Draw the first row of nodes
\foreach \i in {1,...,\n} {
    \node[draw, fill=black, circle, inner sep=1.5pt] (a\i) at (\i, 1) {};
    \node[above=2pt] at (a\i) {};
}

% Draw the second row of nodes
\foreach \i in {1,...,\n} {
    \node[draw, fill=black, circle, inner sep=1.5pt] (b\i) at (\i, 0) {};
    \node[below=2pt] at (b\i) {};
}

\draw[thick] (a1) .. controls +(0,-0.5) and +(-.5,0.5) .. (b3);
\draw[thick] (a1) .. controls +(0.5,-0.5) and +(-.5,+0.5) .. (b4);
\draw[thick] (a2) .. controls +(0.5,-0.5) and +(0,0.5) .. (b4);
\draw[thick] (a3) .. controls +(0,-0.5) and +(0,0.5) .. (b1);

\end{tikzpicture}
\hspace{5pt}
& \hspace{15pt}
\begin{tikzpicture}[scale=.5,every node/.style={scale=.75}]

% Number of nodes
\def\n{4}

% Draw the first row of nodes
\foreach \i in {1,...,\n} {
    \node[draw, fill=black, circle, inner sep=1.5pt] (a\i) at (\i, 1) {};
    \node[above=2pt] at (a\i) {};
}

% Draw the second row of nodes
\foreach \i in {1,...,\n} {
    \node[draw, fill=black, circle, inner sep=1.5pt] (b\i) at (\i, 0) {};
    \node[below=2pt] at (b\i) {};
}

\draw[thick] (a2) .. controls +(0,-0.5) and +(0,0.5) .. (b3);
\draw[thick] (a3) .. controls +(0,-0.5) and +(0,+0.5) .. (b4);
\draw[thick] (a4) .. controls +(0,-0.5) and +(0,0.5) .. (b1);

\end{tikzpicture}

%% file: Ex_after_definition.tex
\begin{center} 
\begin{tikzpicture}[scale=.5,baseline=5pt]
    \def\n{4}
    \foreach \i in {1,...,\n} {
    \node[draw, fill=black, circle, inner sep=1.1pt] (a\i) at (\i, 1) {};
}
    \foreach \i in {1,...,\n} {
    \node[draw, fill=black, circle, inner sep=1.1pt] (b\i) at (\i, 0) {};
}
    \foreach \i in {1,...,\n} {
    \node[draw, fill=black, circle, inner sep=1.1pt] (c\i) at (\i, -1) {};
}

\draw[thick,blue!70!black] (a4) .. controls +(0,-0.5) and +(0,0.5) .. (b2);
\draw[thick,blue!70!black] (b3) .. controls +(0,-0.5) and +(0,0.5) .. (c1);

\node at (2.5,-2) {$\ell = 1$};

\end{tikzpicture}
\hspace{30pt}
\begin{tikzpicture}[scale=.5,baseline=5pt]
    \def\n{4}
    \foreach \i in {1,...,\n} {
    \node[draw, fill=black, circle, inner sep=1.1pt] (a\i) at (\i, 1) {};
}
    \foreach \i in {1,...,\n} {
    \node[draw, fill=black, circle, inner sep=1.1pt] (b\i) at (\i, 0) {};
}
    \foreach \i in {1,...,\n} {
    \node[draw, fill=black, circle, inner sep=1.1pt] (c\i) at (\i, -1) {};
}
\draw[thick,blue!70!black] (a4) .. controls +(0,-0.5) and +(0,0.5) .. (b2);
\draw[thick,red] (a3) .. controls +(0,-0.5) and +(0,0.5) .. (b3);
\draw[thick,blue!70!black] (b3) .. controls +(0,-0.5) and +(0,0.5) .. (c1);
\draw[thick,red] (b2) .. controls +(0.5,-0.5) and +(0,0.5) .. (c4);

\node at (2.5,-2) {$\ell = 2$};

\end{tikzpicture}
\hspace{30pt}
\begin{tikzpicture}[scale=.5,baseline=5pt]
    \def\n{4}
    \foreach \i in {1,...,\n} {
    \node[draw, fill=black, circle, inner sep=1.1pt] (a\i) at (\i, 1) {};
}
    \foreach \i in {1,...,\n} {
    \node[draw, fill=black, circle, inner sep=1.1pt] (b\i) at (\i, 0) {};
}
    \foreach \i in {1,...,\n} {
    \node[draw, fill=black, circle, inner sep=1.1pt] (c\i) at (\i, -1) {};
}
\draw[thick,blue!70!black] (a4) .. controls +(0,-0.5) and +(0,0.5) .. (b2);
\draw[thick,red] (a3) .. controls +(0,-0.5) and +(0,0.5) .. (b3);
\draw[thick,green!70!black] (a2) .. controls +(0,-0.5) and +(0,0.5) .. (b1);
\draw[thick,blue!70!black] (b3) .. controls +(0,-0.5) and +(0,0.5) .. (c1);
\draw[thick,red] (b2) .. controls +(0.5,-0.5) and +(0,0.5) .. (c4);
\draw[thick,green!70!black] (b1) .. controls +(0,-0.5) and +(-.5,0.5) .. (c3);

\node at (2.5,-2) {$\ell = 3$};

\end{tikzpicture}
\end{center}

%% file: Big_bijection_diagram.tex
\begin{equation}
\label{dist density bijection}
    \begin{tikzpicture}[baseline]
    \node (N) at (0,0) {$\N^{n \times n}$};
    \node (A) at (5,0) {$\left\{ A \in \N^{(n+1) \times (n+1)} : A = \scalebox{.5}{$\begin{bNiceArray}{c|ccc}
  \phantom{*} & * & * & * \\
   & * & * & * \\
   & * & * & * \\
   \hline
   & & &
\end{bNiceArray}$}\right\}$};
\node (M) at (10,0) {$\M^0_n$};

\draw [->] (N) -- (A) node [midway,above=1pt,scale=.65] {${\rm L}$};
\draw [->] ([yshift=1mm]A.east) -- ([yshift=1mm]M.west) node [midway,above=1pt,scale=.75] {$\Sigma$};
\draw [<-] ([yshift=-1mm]A.east) -- ([yshift=-1mm]M.west) node [midway,below=1pt,scale=.65] {$\Delta$};
% \draw [->] (N.south) -- (M.south) node [midway,above=1pt, scale=.65] {$\Phi = \Sigma \circ {\rm L}$};
\path
    (N.south) edge[bend right=20,->] node [midway,above=1pt, scale=.65] {$\Phi \coloneqq \Sigma \circ {\rm L}$} (M.south);
    \end{tikzpicture}
\end{equation}

%% file: planepartition.tex
\newcounter{x}
\newcounter{y}
\newcounter{z}
% The angles of x,y,z-axes
\newcommand\xaxis{210}
\newcommand\yaxis{-30}
\newcommand\zaxis{90}
% The top side of a cube
\newcommand\topside[3]{
  \fill[fill=white, draw=black,shift={(\xaxis:#1)},shift={(\yaxis:#2)},
  shift={(\zaxis:#3)}] (0,0) -- (30:1) -- (0,1) --(150:1)--(0,0);
}
% The left side of a cube
\newcommand\leftside[3]{
  \fill[fill=lightgray, draw=black,shift={(\xaxis:#1)},shift={(\yaxis:#2)},
  shift={(\zaxis:#3)}] (0,0) -- (0,-1) -- (210:1) --(150:1)--(0,0);
}
% The right side of a cube
\newcommand\rightside[3]{
  \fill[fill=white, draw=black,shift={(\xaxis:#1)},shift={(\yaxis:#2)},
  shift={(\zaxis:#3)}] (0,0) -- (30:1) -- (-30:1) --(0,-1)--(0,0);
}
% The cube 
\newcommand\cube[3]{
  \topside{#1}{#2}{#3} \leftside{#1}{#2}{#3} \rightside{#1}{#2}{#3}
}
% Definition of \planepartition
% To draw the following plane partition, just write \planepartition{ {a, b, c}, {d,e} }.
%  a b c
%  d e
\newcommand\planepartition[1]{
 \setcounter{x}{-1}
  \foreach \a in {#1} {
    \addtocounter{x}{1}
    \setcounter{y}{-1}
    \foreach \b in \a {
      \addtocounter{y}{1}
      \setcounter{z}{-1}
      \foreach \c in {1,...,\b} {
        \addtocounter{z}{1}
        \cube{\value{x}}{\value{y}}{\value{z}}
      }
    }
  }
}
\begin{tikzpicture}[scale=.4,baseline=(current bounding box.center)]
\planepartition{{8,5,2,1},{5,3,1},{3,2},{1,1}}
\node at (-4,-3) {south};
\node at (4,-3) {east};
\end{tikzpicture}

%% file: Bijection_diagram.tex
\begin{center}
    \begin{tikzpicture}
    \node (D) at (0,0) {$\mathscr{D}_k$};
    \node (P) at (-1.5,.5) {$\mathscr{P}_k$};
    \node (R) at (-1.5,-.5) {$\mathscr{R}_k$};
    \node (M) at (1.9,0) {$\mathscr{G}^k_{1,1}(\infty)$};
    
    \draw [<->] (P) -- (D) node [midway,above=1pt,scale=.65] {(1)};
    \draw [<->] (R) -- (D) node [midway,below=1pt,scale=.65] {(3)};
    \draw [->] (D) -- (M) node [midway,above=1pt,scale=.65] {$\overline{\Sigma}$};
    
\end{tikzpicture}
\end{center}